\documentclass{birkjour}

\usepackage{amssymb}
 \newtheorem{thm}{Theorem}[section]
 \newtheorem{cor}[thm]{Corollary}
 \newtheorem{lem}[thm]{Lemma}
 \newtheorem{prop}[thm]{Proposition}
 \theoremstyle{definition}
 \newtheorem{defn}[thm]{Definition}
 \theoremstyle{remark}
 \newtheorem{rem}[thm]{Remark}
 
 \numberwithin{equation}{section}

\begin{document}

%-------------------------------------------------------------------------
% editorial commands: to be inserted by the editorial office
%
%\firstpage{1} \volume{228} \Copyrightyear{2004} \DOI{003-0001}
%\seriesextraline{This is the Concrete Title of this Book\br H.E. R and S.T.C. W, Eds.}
%\firstpage{1}

\title[Decomposition of the set of Banach limits]
 {Decomposition of the set of Banach limits into discrete and continuous subsets}

%----------Author 1
\author[N. Avdeev]{Nikolai Avdeev}

\address{%
Faculty of Mathematics, Voronezh State University, Voronezh, 394018, Russia}

\email{nickkolok@mail.ru}

\thanks{The work of the first, third and fourth authors was supported by the Theoretical Physics and Mathematics Advancement Foundation ``BASIS''. The work of the second authors was supported by RSF, grant no. 24-21-00220.}
%----------Author 2
\author[E. Semenov]{Evgenii Semenov}
\address{Faculty of Mathematics, Voronezh State University, Voronezh, 394018, Russia}
\email{semenov@math.vsu.ru}

\author[A. Usachev]{Alexandr Usachev}

\address{%
	School of Mathematics and Statistics, Central South University, Hunan, 410085, China}

\email{dr.alex.usachev@gmail.com}

%\thanks{This work was completed with the support of the Theoretical Physics and Mathematics Advancement Foundation ``BASIS''.}
%----------Author 2
\author[R. Zvolinskiy]{Roman Zvolinskii}
\address{Faculty of Mathematics, Voronezh State University, Voronezh, 394018, Russia}
\email{roman.zvolinskiy@gmail.com}
%\thanks{This work was completed with the support of the Theoretical Physics and Mathematics Advancement Foundation ``BASIS''.}



\date{January 1, 2004}
\dedicatory{To Professor Fedor Sukochev on the occasion of his 60-th birthday}
%%% ----------------------------------------------------------------------

\begin{abstract}
The aim of this work is to describe subsets of Banach limits in terms of a certain functional characteristic. We compute radii and cardinalities for some of these subsets.
\end{abstract}

%%% ----------------------------------------------------------------------
\maketitle
%%% ----------------------------------------------------------------------
%\tableofcontents
\section{Introduction and preliminaries}

Undoubtedly, singular traces is one of the Fedor Sukochev areas of interest. A second edition of the ``Singular traces'' book (co-authored by him) further supports this statement. 
Singular traces have a fairly long history of interconnection with Banach limits. In the late 90s F. Sukochev with coauthors employed Banach limits in the construction of Dixmier traces on Lorentz spaces \cite{dodds1998symmetric}. In its neatest form this relation was explained in \cite{semenov2015banach}, where it was proved that the set of Banach limits and that of positive normalised traces on the weak-trace class ideal (of compact operators on a Hilbert space) are isometrically isomorphic. In this paper we present new results on the structure of Banach limits. 

Denote by $\ell_\infty$ the set of all bounded sequences $x = (x_1, x_2, \ldots)$ equipped with the uniform norm
$$
\|x\|_{\ell_\infty} = \sup_{n \in \mathbb N} |x_n|
$$
and the usual partial order. Here $\mathbb N$ denotes the set of natural numbers. By $\ell^*_\infty$ we denote the dual of $\ell_\infty$.

\begin{defn}
	A linear functional $B \in \ell^*_\infty$ is said to be a Banach limit if it satisfies the following conditions:
	\begin{enumerate}
	\item[(i)]	$B \ge 0$, that is $Bx \ge 0$ for all $x \in \ell_\infty$ such that $x \ge 0$;
	\item[(ii)] $B\mathrm{1\negthickspace I} = 1$, where $\mathrm{1\negthickspace I} = (1, 1, \dots)$;
		\item[(iii)]
		$BTx = Bx$ for all $x \in \ell_\infty$, where $T$ is the left-shift operator, i.e., $$T(x_1, x_2, x_3, \dots) = (x_2, x_3, x_4, \dots).$$
	\end{enumerate}
\end{defn}
	
	The existence of Banach limits was first announced by S. Mazur~\cite{mazur1929metodach}. The proof using the Hahn-Banach theorem was first presented in the book by S. Banach~\cite{banah2001theory}. From the definition, it follows that
	$$
	\underset{n \to \infty}{\lim\inf} \, x_n \le Bx \le 
	\underset{n \to \infty}{\lim\sup} \, x_n
	$$
	for all $x \in \ell_\infty$, and consequently,
	$$
	Bx = \lim_{n \to \infty} x_n
	$$
	for any convergent sequence. Furthermore,
	$\|B\|_{\ell^*_\infty} = 1$ for any $B \in \mathfrak B$, where $\mathfrak B$
	denotes the set of Banach limits. It is easy to see that $\mathfrak B$ is a closed convex set on the unit sphere of the space $\ell^*_\infty$. G. G. Lorentz proved that for given $x \in \ell_\infty$ and $\lambda \in \mathbb R^1$, the identity $Bx = \lambda$ holds 
	for all $B \in \mathfrak B$ if and only if
$$
\lim_{n \to \infty} \frac 1n \sum_{k = m + 1}^{m + n} x_k = \lambda
$$
uniformly for $m \in \mathbb N$. In this case, the sequence
$x = (x_1, x_2, \ldots)$ is said to be almost convergent to $\lambda \in \mathbb R^1$~\cite{lorentz1948contribution}. The set of all sequences almost converging to $\lambda \in \mathbb R^1$ is denoted by $ac_\lambda$.

For a bounded linear operator $H$ on $\ell_\infty$, a Banach limit $B \in \mathfrak B$ is said to be $H$-invariant if $Bx = BHx$ for all $x \in \ell_\infty$. Invariant Banach limits were first considered in the paper by R. Agnew and A. Morse~\cite{agnew1938extensions}. Later, W. Eberlein proved the existence of Banach limits invariant under regular Hausdorff transformations~\cite{eberlein1950banach}. This approach was further developed in~\cite{semenov2010invariant}. If the operator $H$ satisfies the following conditions:
	\begin{enumerate} 
	\item
	$H \ge 0$ and $H\mathrm{1\negthickspace I} =
	\mathrm{1\negthickspace I}$, 
	\item
	$Hc_0 \subset c_0$,
	\item
	$\underset{j \to \infty}{\lim\sup} \, (A(I - T)x)_j
	\ge 0$ for all $x \in \ell_\infty$,
	$A \in R(H) = \mathrm{conv} \left\{H^k, \, k \in \mathbb N\right\}$, 
\end{enumerate} 
then there exists a $H$-invariant Banach limit. We denote the set of all such Banach limits by $\mathfrak B(H)$. It can be shown that $\mathfrak B(H)$ is a closed convex subset of $\mathfrak B$. The Ces\`aro operator 
$$
(Cx)_n = \frac 1n \sum_{k = 1}^n x_k, \, n \in \mathbb N
$$
and dilation operators
$$
(\sigma_n x) = (\underbrace{x_1, x_1, \ldots, x_1}_n,
\underbrace{x_2, x_2, \ldots, x_2}_n, \ldots), \,
n \in \mathbb N
$$
satisfy conditions 1-3. Therefore, the sets $\mathfrak B(C)$ and $\mathfrak B(\sigma_n)$ are non-empty for any $n \in \mathbb N.$

It is well-known that the space $\ell_\infty$ is lattice isometric to the space $C(\beta\mathbb N)$ of continuous functions on the Stone-{\v C}ech compactification of the set $\mathbb N$, and therefore, it is an $AM$-space with unity \cite[Theorem 12.28]{aliprantis2006positive}. Consequently, its dual $\ell_\infty^*$ is an $AL$-space. Thus, it is lattice isometric to the space $L_1 (\Omega)$ for some set $\Omega$ with measure $\mu$ \cite[Theorem 12.26]{aliprantis2006positive}. Let $\Omega=\Omega_d \cup \Omega_c$ be a partition of the set $\Omega$ into measurable subsets such that the measure $\mu$ restricted to $\Omega_d$ (resp., to $\Omega_c$) is discrete (resp., continuous). Then,
	$$
\ell^*_\infty \approx L_1(\Omega_d) \oplus L_1(\Omega_c).
$$

The kernel of $N(T^*-I)$ of the operator $T^*-I$ is a Riesz space in $\ell_\infty^*$ and, thus, it is an $AL$-space itself. Therefore, a similar representation holds
$$
N(T^*-I) \approx L_1(\Omega_d^1) \oplus L_1(\Omega_c^1).
$$

The space $N(T^*-I)$ consists of all linear functionals on $\ell_\infty$ that are invariant under the shift operator. Therefore, the positive part of its unit sphere $S^+_{N(T^*-I)}$ is nothing but the set $\mathfrak B$ of Banach limits. Thus, the set $\mathfrak B$ can be decomposed into a direct sum:
	$$
\mathfrak B = \mathfrak B_d \oplus \mathfrak B_c,
$$
where $\mathfrak B_d \approx S^+_{L_1(\Omega_d^1)}$ and $\mathfrak B_c \approx S^+_{L_1(\Omega_c^1)}$.
Sets $\mathfrak B_d$ and $\mathfrak B_c$ are referred to as the discrete and continuous subsets of $\mathfrak B$, resp. For any $B \in \mathfrak B$, there exist $B_1 \in \mathfrak B_d$, $B_2 \in \mathfrak B_c$, and $\lambda \in [0, 1]$ such that:
	$$
B = (1 - \lambda) B_1 + \lambda B_2.
$$

According to the Krein-Milman theorem:
	$$
\mathfrak B = \overline{\operatorname{conv}} \operatorname{ext} \mathfrak B,
$$
where $\operatorname{ext} \mathfrak B$ is the set of extreme points of $\mathfrak B$, and the closure of the convex hull is taken in the weak$^*$ topology. 
%As demonstrated by C. Chou~\cite{chou1969size}, the set $\operatorname{ext} \mathfrak B$ has a cardinality of $2^{\mathfrak c}$, where $\mathfrak c$ is the cardinality of the continuum. 
For any distinct $B_1, B_2, \ldots \in \operatorname{ext} \mathfrak B$, and $\lambda_1, \lambda_2, \ldots \in \ell_1$ the following equality holds:
	$$
\left\|\sum_{k = 1}^\infty \lambda_k B_k\right\|_{\ell^*_\infty} = \sum_{k = 1}^\infty |\lambda_k|.
$$

The main properties of the set $\mathfrak B$ are discussed in the surveys~\cite{semenov2023banach, sofi2021banach, semenov2020geom}, as well as in the monograph~\cite{das2021banach}. Some of the results in this article were previously announced in~\cite{semenov2023banach1}.

\section{A functional characteristic}

For every $B \in \mathfrak B$ let us consider a function defined on $[0, 1]$ by setting:
$$
\gamma(B, t) = B\Bigg(\bigcup^\infty_{n = 1}
[2^n, 2^{n + t})\Bigg),
$$
where $\bigcup^\infty_{n = 1} [2^n, 2^{n + t})$ stands for a sequence $x_k$ equal to
 $1$ for $2^n \le k < 2^{n + t}$, $k \in \mathbb N$ and $0$ otherwise.
This function first appeared in~\cite{semenov2019main}.

Let $\Gamma$ be the set of all non-decreasing functions $f$ on $[0,1]$ such that $f(0)=0$ and $f(1)=1.$ For every $f\in \Gamma$, there exists $B\in \mathfrak B$ such that $\gamma(B,\cdot)=f$~\cite[Lemma 1]{semenov2019main}. For each $f\in \Gamma$, define
$$\mathfrak B(f) = \{B\in \mathfrak B: \, \gamma(B, \cdot) = f\}.$$
It is easy to see that $\mathfrak B(f)$ is a weak$^*$-closed convex subset of $\mathfrak B$.

	\begin{thm}\label{thmcont}
	A Banach limit $B$ belongs to the set $\mathfrak B_c$ if and only if the function
	$\gamma(B, \cdot)$ is continuous.	
\end{thm} 

\begin{proof}
	Suppose, to the contrary, that $B\in \mathfrak B_c$ and the function $\gamma(B, \cdot)$ is discontinuous. Since it is non-decreasing, it has a jump $a > 0$ at some point $s \in [0, 1]$. For the sake of clarity, assume that $s \in (0, 1)$ and that $\gamma(B, \cdot)$ is continuous from the right at $s$. Then
		\begin{equation*} \label{eq:1}
		\gamma(B, t) = af(t) + (1 - a) g(t),
	\end{equation*}
	where
	$$
	f(t) = \left\{
	\begin{array}{cl}
		0, & 0 \le t < s \\
		1, & s \le t \le 1.
	\end{array}
	\right. 
	$$
	
	We shall show that
		$$
	\operatorname{ext} \mathfrak B(f) = \operatorname{ext}
	\mathfrak B \cap \mathfrak B(f).
	$$
	To this end, it is sufficient to show that
	\begin{equation} \label{eq:2}
		\operatorname{ext} \mathfrak B(f) \subset
		\operatorname{ext} \mathfrak B. 
	\end{equation}
	Suppose $B_0 = \frac 12 (B_1 + B_2)$, where $B_0 \in \operatorname{ext} \mathfrak B(f)$, and $B_1, B_2 \in \mathfrak B$. From this, we have $f(t) = \gamma(B_0, t) = \frac 12 \big(\gamma(B_1, t) + \gamma(B_2, t)\big)$. Since $\gamma(B_0, t)$ is an extreme point of $\Gamma$, it follows that $\gamma(B_0, t) = \gamma(B_1, t) = \gamma(B_2, t)$. This implies that $B_1$ and $B_2$ are in $\mathfrak B(f)$. Since $B_0 \in \operatorname{ext} \mathfrak B(f)$, it follows that $B_0 \in \operatorname{ext} \mathfrak B$, which proves (\ref{eq:2}).
	
	From \eqref{eq:2}, it follows that
		$$
	\mathfrak B(f) \subset \overline{\operatorname{conv}}
	\operatorname{ext} \mathfrak B = \mathfrak B_d,
	$$
	where the closure is taken in the $\ell_\infty$ norm. In particular, $B_0$ belongs to a convex combination of a finite or countable number of $B_3, B_4, \ldots \in \operatorname{ext} \mathfrak B$. Then, for a sufficiently small $\alpha>0$, we have $\|B_0 - \alpha B_3\|_{\ell^*_\infty} < 1$. From the assumption that $B_0 \in \mathfrak B_c$, it follows that $B_0 \perp \mathfrak B_d$ \cite[Theorem 25]{alekhno2015banach}, and $\|B_0 - \alpha B_3\|_{\ell^*_\infty} = 1 + |\alpha|$ for all $\alpha$. The obtained contradiction proves that $B \in \mathfrak B_c$ implies the continuity of the function $\gamma(B, \cdot)$.
	
	Conversely, let $\gamma(B, \cdot)$ be continuous. It follows from \cite[Theorem 25]{alekhno2015banach}, that a Banach limit $B$ belongs to the set $\mathfrak B_c$ if and only if the function $t\mapsto B\{k\in \mathbb N: x_k\le t\}$ is continuous for some sequence $x\in \ell_\infty$.
	
	Set $x_1=2$ and $x_k=\log_2 k -n$ for $2^n\le k< 2^{n+1}$, $n=1,2, \dots$. Then
	$$\{k\in \mathbb N: x_k\le t\}= \bigcup_{n\ge 1} \{k\in \mathbb N: \log_2 k -n \le t\}= \bigcup_{n\ge1} [2^n, 2^{n+t}].$$
	
	Since the sets $\bigcup_{n\ge1} [2^n, 2^{n+t})$ and $\bigcup_{n\ge1} [2^n, 2^{n+t}]$ differ only by the set $\{2^{n+t}\}_{n=1}^\infty$ (the characteristic function of which almost converges to zero), for all $t\in[0,1]$ we have $\gamma(B,t)=B\{k\in \mathbb N: x_k\le t\}$ for the chosen sequence $x\in \ell_\infty$. The claim follows from the continuity of the function $\gamma(B,\cdot)$.
\end{proof}

\begin{thm}\label{thmdiscr}
	A Banach limit $B$ belongs to the set $\mathfrak B_d$ if and only if the function
	$\gamma(B, \cdot)$ belongs to the closed convex hull of functions $\varphi_s$, where
	$$
	\varphi_s (t) = \left\{
	\begin{array}{cc}
		0, & 0 \le t < s \\
		1, & s \le t \le 1
	\end{array}
	\right.
	\text{ or }
	\left\{
	\begin{array}{cl}
		0, & 0 \le t \le s \\
		1, & s < t \le 1.
	\end{array}
	\right.
	$$
	($0 < s \le 1$ in the first case and $0 \le s < 1$ in the second one).	
\end{thm} 

\begin{proof}
	If $B \in \mathfrak B_d$, then
	$$
	B = \sum_{i = 1}^\infty \alpha_i B_i
	$$ 
	for some
	$B_i \in \operatorname{ext \mathfrak B}$, $\alpha_i \ge 0$,
	$\sum_{i = 1}^\infty \alpha_i = 1$. Hence,
	$$
	\gamma(B, t) = 
	\Bigg(\sum_{i = 1}^\infty \alpha_i B_i\Bigg)
	\Bigg(\bigcup^\infty_{n = 1} [2^n, 2^{n + t})\Bigg) =
	\sum_{i = 1}^\infty \alpha_i
	B_i \Bigg(\bigcup^\infty_{n = 1} [2^n, 2^{n + t})\Bigg).
	$$
	By \cite[Corollary 29]{semenov2014geom} we have that
	$B_i \big(\bigcup^\infty_{n = 1} [2^n, 2^{n + t})\big) = 0$
	or $1$ for every $i = 1, 2, \ldots$ and
	for every $t \in [0, 1]$. Thus, $B_i \big(\bigcup^\infty_{n = 1}
	[2^n, 2^{n + t})\big) = \varphi_{s_i} (t)$ and 
	$$
	\gamma(B, t) = \sum_{i = 1}^\infty \alpha_i \varphi_{s_i} (t).
	$$
	This proves the first part of Theorem.
	
	As it was proved in the first part of the proof of Theorem \ref{thmcont}, if the equality
	$\gamma(B, t) = \varphi_s (t)$ holds for some $s \in [0, 1]$ and all
	$t \in [0, 1]$, then $B \in \mathfrak B_d$. Let $\gamma(B, \cdot)$
	be a convex combination of a finite number of functions $\varphi_{s_i}$,
	$1 \le i \le m$. Then $B$ is a convex combination of finite number of Banach limits from $\mathfrak B_d$. Since $\mathfrak B_d$ is convex, it follows that $B$ belongs to $\mathfrak B_d$. If
	$\gamma(B, \cdot)$ is a convex combination of a countable number of functions
	$\varphi_{s_i}$, then $B \in \mathfrak B_d$, since the set
	$\mathfrak B_d$ is closed.	
\end{proof} 

Thus, for every $B \in \mathfrak B_d$, the function $\gamma(B, \cdot)$ has at most countable number of discontinuities. At each such point, the function $\gamma(B, \cdot)$ may be discontinuous from the left, from the right, or both.

For any $B \in \mathfrak B$, the function $\gamma(B, \cdot)$ is non-decreasing. Since a non-decreasing function is differentiable almost everywhere, then it makes sense to speak about the derivative of the function $\gamma(B, \cdot)$.

\begin{thm}\label{thmder}
	If $B \in \mathfrak B_d$, then $\gamma(B, t)' = 0$
	for almost every $t\in[0,1]$.
\end{thm} 

\begin{proof}
The assertion holds when $B$ is a convex combination of a finite number of extreme points of $\mathfrak B$, that is
	$$
	B = \sum_{i = 1}^m \lambda_i B_i, \
	\lambda_i > 0, \
	\sum_{i = 1}^m \lambda_i = 1, \
	B_i \in \operatorname{ext} \mathfrak B.
	$$
 Using this statement for an arbitrary $B \in \mathfrak B_d$, we have
	$$
	\frac{d\gamma}{dt}(B, t) = \frac{d\gamma}{dt}
	\bigg(\sum_{i = m}^\infty \lambda_i B_i, t\Bigg)
	$$
up to a finite number of points and for any $m \in \mathbb N$. By \cite[Chapter 8, Section 2, Theorem 5]{natanson1957}, the inequalities
	
	\begin{align*}
		0 \le \int\limits_0^1 \gamma
		\Bigg(\sum_{i = m}^\infty \lambda_i B_i, t\Bigg)' dt &\le
		\gamma \Bigg(\sum_{i = m}^\infty \lambda_i B_i, 1\Bigg) -
		\gamma \Bigg(\sum_{i = m}^\infty \lambda_i B_i, 0\Bigg) \\
		&\le
		\left\|\sum_{i = m}^\infty \lambda_i B_i\right\|_{\ell^*_\infty}
		\le \sum_{i = m}^\infty \lambda_i
	\end{align*}		 
	hold for each $m \in \mathbb N$. Therefore,
		$$
		\int\limits_0^1 \gamma(B, t)' dt = 0.$$

	Since the function $\gamma(B, \cdot)$ is non-negative, we conclude that $\gamma(B, t)' = 0$ almost everywhere.
\end{proof} 

The converse of Theorem \ref{thmder} does not hold. Recall that a continuous, non-decreasing, and non-constant function $h$ is called singular if $h' = 0$ almost everywhere. The most well-known example of such a function is the Cantor function. If $h(0) = 0$, $h(1) = 1$, and $h$ is monotone on $[0, 1]$, then by \cite[Proposition 2]{semenov2019main}, there exists $B \in \mathfrak B$ such that $\gamma(B, t) = h(t).$ In particular, it can be additionally assumed that $h$ is singular. In this case, Theorem \ref{thmcont} implies that $B \in \mathfrak B_c$.
Thus, we the following result holds:

\begin{cor}\label{corsing}
	If $B \in \mathfrak B$ and $\gamma(B, \cdot)$ is a singular function, then $B \in \mathfrak B_c$.
\end{cor} 

\begin{rem}
	For every $a>1$ let us introduce a function 
	$$\gamma_a(B, t)=B \left(\bigcup_{n=1}^\infty [a^n, a^{n + 1}) \right), \ t\in[0,1].$$ 
	Theorems \ref{thmcont}-\ref{thmder} and Corollary \ref{corsing} hold for the function $\gamma_a(B, \cdot)$ as well. It follows from \cite[Proposition 20]{semenov2019main} that $\gamma_a(B, t)=t$ for every $a>1$  and $t\in[0,1]$, provided that $B\in \mathfrak B(C)$.
\end{rem}

\section{Ces\`aro invariant Banach limits}

In this section, we will establish several new properties of the set $\mathfrak B(C)$ of Banach limits invariant under the Ces\`aro operator. As it was mentioned in Introduction, this set is non-empty. However, not every Banach limit is Ces\`aro invariant. On the other hand every Banach limit is trivially invariant under the operator $A$ if $A - I : \ell_\infty \to ac_0$. It is worth noting that for any Banach limit, there exists a nontrivial operator $G_B : \ell_\infty \to \ell_\infty$ such that
	$\mathfrak B(G_B) = \{B\}$. Indeed, one can take $G_Bx:= Bx \cdot \mathrm{1\negthickspace I}.$

Let $\{n_k\}^\infty_{k = 1}$ be a strictly increasing sequence of natural numbers. We denote by $V$ the set of sequences of the form:
\begin{equation} \label{eq:11}
	x = (x_i) = \left\{
	\begin{array}{cr}
		1, & n_{2k} \le i < n_{2k + 1} \\
		0, & n_{2k + 1} \le i < n_{2k + 2}
	\end{array}
	\right., \ i, k \in \mathbb N, 
\end{equation}
where $\{n_k\}^\infty_{k = 1}$ satisfies the condition:
$$
\lim_{k \to \infty} (n_{k + 1} - n_k) = \infty.
$$
By $V_0$ we denote is a subset of $V$ where a stronger condition:
\begin{equation} \label{eq:33}
	\lim_{k \to \infty} n_{k + 1} / n_k = \infty
\end{equation}
holds.

\begin{lem}\label{lemBCorbit}
	If $x \in V_0$, then
	\begin{equation*} \label{eq:22}
		\{Bx : \, B \in \mathfrak B(C)\} = [0, 1].
	\end{equation*}	
\end{lem} 

\begin{proof}
	The assumption \eqref{eq:33}
	implies
	\begin{align*}
		&\underset{j \to \infty}{\lim\inf} \, (C^m x)_j = 0, \quad \underset{j \to \infty}{\lim\sup} \, (C^m x)_j = 1 
	\end{align*}
	for every $m \in \mathbb N$. It was proved in~\cite{alehno2016por}, that
	\begin{equation*} \label{eq:66}
		\{By: \, B \in \mathfrak B(C)\} =
		\left[\lim_{m \to \infty} \underset{j \to \infty}{\lim\inf} \, (C^m y)_j,
		\lim_{m \to \infty} \underset{j \to \infty}{\lim\sup} \, (C^m y)_j\right].
	\end{equation*}
	The assertion follows from these three equalities. 	
\end{proof}

The result in the opposite direction can be found in \cite[Theorem~15]{semenov2010invariant}. Setting $n_k = 2^k$ for $k \in \mathbb N$ in the notations of \eqref{eq:11}, we obtain
$$
\{Bx: \, B \in \mathfrak B(C)\} = \{1/2\}.
$$

This lemma also shows that the assumption \eqref{eq:33} cannot be replaced by a weaker one:
$$
\lim_{k \to \infty} (n_{k + 1} - n_k) = \infty.
$$
Therefore, $V_0 \subsetneq V$.
Also, if $x \in V$, then the Sucheston theorem~\cite{sucheston1967banach} implies
$$
\{Bx: \, B \in \mathfrak B\} = [0, 1].
$$

The following result is a particular form of \cite[Lemma 13]{semenov2014geom}.
\begin{lem}\label{lemBV}
	For every $B \in \mathfrak B$
	\begin{equation*} \label{eq:77}
		0, 1 \in \{Bx: \, x \in V_0\}.
	\end{equation*}	
\end{lem}

\begin{cor}
	The set $\{Bx: \, x \in V_0\}$ is finite or countable for every $B \in \mathfrak B_d$ and $x \in V_0$.	
\end{cor}

	It follows from~\cite{semenov2020geom}, that both diameter and radius of the set  $\mathfrak B$ in $\ell^*_\infty$ equal $2$, i.e., 
$$
d\big(\mathfrak B,\ell^*_\infty\big) :=\sup_{x, y \in \mathfrak B} \|x - y\|_{\ell^*_\infty} = 2; \quad r\big(\mathfrak B,\ell^*_\infty\big) :=
\inf_{x \in \mathfrak B} \sup_{y \in \mathfrak B} \|x - y\|_{\ell^*_\infty} = 2.
$$
Furthermore, $d\big(\mathfrak B(C),\ell^*_\infty\big) =d\big(\mathfrak B(\sigma_n),\ell^*_\infty\big) = 2$ ~\cite{alehno2016por}. 

For every $B\in \mathfrak B(C)$ one has $\gamma(B,t)=t$ for every $t\in[0,1]$. Hence, $\mathfrak B(C)\subset \mathfrak B(f_1)$, where $f_1(t)=t$ for every $t\in[0,1]$. Thus, $d\big(\mathfrak B(f_1),\ell^*_\infty\big) = 2$. It was shown in \cite[Theorem 25]{semenov2019main} that the inclusion $\mathfrak B(C)\subset \mathfrak B(f_1)$ is proper, that is there exists $B\in \mathfrak B(f_1)\setminus \mathfrak B(C)$.

The following result shows that the radius of the set $\mathfrak B(C)$ is also equal to 2.

\begin{thm}\label{thmBCrad}
	$r\big(\mathfrak B(C), \ell^*_\infty\big) = 2$.
\end{thm}

\begin{proof}
Let $B_1 \in \mathfrak B(C)$. According to Lemma \ref{lemBV}, you can find a sequence $x \in V_0$ such that $B_1 x = 0$. Due to Lemma \ref{lemBCorbit}, there exists $B_2 \in \mathfrak B(C)$ such that $B_2 x = 1$. Then, $(B_2 - B_1) x = 1 - 0 = 1$.

The sequence $y = 2x - \mathrm{1\negthickspace I}$ obviously has the property that all of its coordinates are $\pm 1$. Therefore, $\|y\|_{\ell_\infty} = 1$, and
	$$
	(B_2 - B_1) y = (B_2 - B_1) (2x - \mathrm{1\negthickspace I}) = 2 - 0 = 2.
	$$
	Hence,
	$$
	r\big(\mathfrak B(C), \ell^*_\infty\big) \geqslant \|B_2 - B_1\|_{\ell^*_\infty} \geqslant 2.
	$$
	The converse inequality is obvious.
\end{proof}

Now we shall apply the properties of the set $\mathfrak B(C)$, obtained in Lemma \ref{lemBCorbit}, to study the set $\mathfrak B_c$, the continuous part of the set $\mathfrak B$.

	\begin{lem}\label{lemBcorbit}
	For every $x \in V_0$ we have
	$$
	\{Bx: \, B \in \mathfrak B_c\} = [0, 1].
	$$	
\end{lem}

\begin{proof}
It was proved in~\cite{semenov2013extreme} that sets $\operatorname{ext} \mathfrak B$ and $\mathfrak B(C)$ are orthogonal. Therefore, $\mathfrak B(C)$ is contained in the orthogonal complement of $\mathfrak B_d$, i.e., $\mathfrak B(C) \subset \mathfrak B_c$.
Applying Lemma \ref{lemBCorbit}, we obtain:
	$$
	[0, 1] = \{Bx: \, B \in \mathfrak B(C)\} \subset \{Bx: \, B \in \mathfrak B_c\}.
	$$
	The converse inclusion is evident.	
\end{proof}

Further, we aim to compute the radius of the set $\mathfrak B_c$. Before we proceed to the result we explain why Theorem \ref{thmBCrad} cannot be applied.

It is obvious from the definition that $d(M, \ell^*_\infty) \leqslant d(N, \ell^*_\infty)$ for every subsets $M, N \in \ell_\infty^*$ such that $M\subseteq N$. It turns out that a similar statement does not hold for radii of sets, in general. We shall need the definition of relative radius of a set $M$ in a normed space $E$, which is defined as follows:
$$r_E(M)=\inf_{x\in E}\sup_{y\in M} \|x-y\|_E.$$

\begin{prop}
	Let $E$ be a normed space. The inequality $r(M,E)\le r(N,E)$ holds for every convex bounded set $M\subset E$ and every set $N\subset E$ such that $M\subset N$ if and only if $E$ is a 2-dimensional or Hilbert space.
\end{prop}

\begin{proof}
	For every $N\subset E$ such that $M\subseteq N$ one has 
	\begin{equation*}
		\inf_{x\in N}\sup_{y\in N} \|x-y\|_E \ge \inf_{x\in N}\sup_{y\in M} \|x-y\|_E\ge \inf_{x\in E}\sup_{y\in M} \|x-y\|_E.
	\end{equation*}
	Thus,
	\begin{equation}\label{radii}
	r(N,E)\ge r_E(M).
	\end{equation}	
	
	It was proved in \cite{klee1960asymp} that $E$ is a 2-dimensional or a Hilbert space if and only if $r(M,E)=r_E(M)$ for every convex bounded set $M\in E$. 
	
	Let $E$ be a 2-dimensional or a Hilbert space. The assertion follows from \eqref{radii} and the above mentioned result.
	
	Conversely, suppose to the contrary that there exists a convex bounded set $M\in E$ and a set $N\in E$ such that $M\subseteq N$ and $r(M, E)>r(N,E)$. Using \eqref{radii}, we obtain that $r(M,E)>r_E(M).$ The assertion follows from \cite{klee1960asymp}. 
\end{proof}

\begin{thm}
	$r(\mathfrak B_c, \ell^*_\infty) = 2$.	
\end{thm} 

\begin{proof}
The proof follows the same idea as that of Theorem \ref{thmBCrad}. Let $B_1 \in \mathfrak B_c$. By Lemma \ref{lemBV}, there exists a sequence $x \in V_0$ such that $B_1 x = 0$. Using Lemma \ref{lemBcorbit}, we conclude that there exists $B_2 \in \mathfrak B_c$ such that $B_2 x = 1$. Then $(B_2 - B_1) x = 1$. For $y = 2x - \mathrm{1\negthickspace I}$, we have $\|y\|_{\ell_\infty} = 1$, and
	$$
	(B_2 - B_1) y = (B_2 - B_1) (2x - \mathrm{1\negthickspace I}) = 2.
	$$
	Hence, $r(\mathfrak B_c, \ell^*_\infty) = 2$. 	
\end{proof}

\section{Some cardinalities}
 C. Chou proved that the cardinality of the set ${\rm ext}\! \ \mathfrak B$ is $2^{\mathfrak{c}}$ \cite{chou1969size}. This implies that ${\rm card}\! \ (\mathfrak B_d )=2^{\mathfrak{c}}$. 
%	It was shown [DAN, Theorem 5] that for every $B\in \operatorname{ext} \mathfrak B$ the function
%	$$\varphi(B,t)=Bz_t, \ 0<t<1$$
%	takes values $0$ or $1$ only (with a jump at some $0\leqslant s\leqslant 1$). 
In \cite[Corollary 13]{semenov2019main} it was proved that for every $s\in[0,1]$ the set of all $B\in {\rm ext}\! \ \mathfrak B$ such that 
$\gamma(B,\cdot)=\chi_{[s,1]}$
has cardinality $2^{\mathfrak{c}}$. This splits the set ${\rm ext}\! \ \mathfrak B$ (and, so, the set $\mathfrak B_c$) into continually many disjoint sets of maximal cardinality. 

Since the cardinality of $\mathfrak B(C)$ is maximal \cite[Section 4]{alehno2018invariant} and $\mathfrak B(C) \subset \mathfrak B_c$, it follows that  $\mathfrak B_c$ is of maximal cardinality as well. However, the set $\mathfrak B(C)$ rather appears to be a small subset of of $\mathfrak B_c$, since for every $B\in \mathfrak B(C)$ we have $\gamma(B,t)=t, 0\le t\le1$, whereas by Theorem \ref{thmcont} the set $\{\gamma(B,\cdot): B\in \mathfrak B_c\}$ is much more diverse. In this section we address the question of how big is the set $\mathfrak B_c \setminus \mathfrak B(C)$.  We show that for every absolutely continuous function $f\in \Gamma$ with H\"older continuous derivative, the set of all $B\in \mathfrak B$ with $\gamma(B,\cdot)=f$ has cardinality $2^{\mathfrak{c}}$. Hence, the set $\mathfrak B_c \setminus \mathfrak B(C)$ splits into continually many disjoint sets of maximal cardinality.

We will use the following (particular form of the) result proved in \cite[Lemma 2]{alehno2016por}. 

\begin{lem}\label{lem1}
	Let $D$ and $M$ be convex subsets of $\ell_\infty^*$ (equipped with weak$^*$ topology)
	and let $M$ be compact. Suppose that there exist mappings 
	$f_1, f_2: \ell_\infty^* \to \ell_\infty^*$ satisfying the following conditions:
	
	\begin{enumerate}
		\item[(i)] $f_2$ is affine and continuous;
		\item[(ii)] The following inclusion holds:
		$${\rm ext}\! \ D \subseteq f_1^{-1}(M) \cap (f_2\circ f_1 - I)^{-1}(0)\text{;}$$
		\item[(iii)] The following inclusion holds: $f_2(M) \subseteq D$.
	\end{enumerate}
	Then the inequality
	$${\rm card}\! \ ({\rm ext}\! \ M ) \geqslant {\rm card}\! \ ({\rm ext}\! \ D)$$
	holds.
\end{lem}

%	Let $\Gamma$ be the set of all non-decreasing functions on $[0,1]$ such that $f(0)=0$ and $f(1)=1.$

Denote by $EL$ the set of all extended limits on $\ell_\infty$.
For every $f\in \Gamma$ and $A \subset EL$ define the set
$$A(f):= \left\{ \omega \in A : \gamma(\omega,\cdot)=f \right\}.$$

%We compute the cardinality of the set $A(f)$ for a wide class of functions $f\in \Gamma$ and $A$ to be the set of all extended limits and all Banach limits. 
%In the latter case, $F_A(f)$ is denoted by $\mathfrak B(f)$.

For every $n\in \mathbb N$, setting 
$$a_{n,m}= \frac1{\log 2} \sum_{i=2^n+1}^{2^{n}+m} \frac1i, \ m=1, \dots, 2^n,$$
we obtain a partition 
$$0=a_{n,0}< a_{n,1} < a_{n,2} < \dots < a_{n,2^n}<a_{n,2^n+1}=1.$$
of the interval $[0,1]$ into $2^n+1$ intervals of length at most
\begin{equation}\label{partitionmesh}
	\max\{a_{n,1}, 1- a_{n,2^n}\}=o(1/n), \ n\to \infty.
\end{equation}

The choice of such a peculiar partition is justified in Lemma \ref{justification} below.

Fix an absolutely continuous $f\in \Gamma$ and define the operator $Q$ on~$\ell_\infty$ by setting $(Qx)_{1} =x_1$, $(Qx)_{2} =x_2$ and
$$(Qx)_{n+2} = \sum_{k=2^n+1}^{2^{n+1}} w_{n,k} x_k, \ n=1, 2, \dots,$$
where positive weights are defined as follows:
$$w_{n,k}=f'(a_{n,k-(2^n+1)}) \cdot \frac1{k\log 2}.$$
% for $2^n \le k \le 2^{n+1}-1$, $k \neq \lfloor 2^{n+s} \rfloor$ and
%$$w_{n,\lfloor 2^{n+s} \rfloor}=1-2^{-2n}(2^n-1).$$

For every $0<t<1$ define
$$(y_t)_k = \begin{cases}
	1, 2^n \leqslant k < 2^{n+t}\\
	0, 2^{n+t} \leqslant k < 2^{n+1}
\end{cases}, \ k,n \in \mathbb N.$$	
Note, that $\gamma(B, t)=B(y_t).$

\begin{lem}\label{justification}
	For an absolutely continuous function $f\in \Gamma$ and every $0\leqslant t \leqslant 1$ we have $(Qy_t)_n \to f(t)$, $n\to \infty$.
\end{lem}

\begin{proof}
	For $t=0$ we have $y_t\equiv 0$ and, so, $(Qy_t)_n \to 0=f(0)$. Fix $0<t\leqslant 1$. For every $n\in \mathbb N$, setting 
	$$\bar a_{n,m}= a_{n,m}, \ m=0, \dots, \lfloor 2^{n+t} \rfloor - 2^n, \ \text{and}, \ \bar a_{n,\lfloor 2^{n+t} \rfloor - 2^n+1}=t$$
	we obtain a partition
	of the interval $[0,t]$ into $\lfloor 2^{n+t} \rfloor - 2^n+1$ intervals of length at most
	$$\max\{\bar a_{n,1}, t- \bar a_{n,\lfloor 2^{n+t} \rfloor - 2^n}\}=o(1/n), \ n\to \infty.$$
	
	Thus,
	\begin{align*}
		(Qy_t)_{n+2} &= \sum_{k=2^n+1}^{2^{n+1}} w_{n,k} (y_t)_k \\
		&= \sum_{k=2^n+1}^{\lfloor 2^{n+t} \rfloor} f'(a_{n,k-(2^n+1)}) \cdot \frac1{k\log 2} \to \int_0^t f'(s) \, ds = f(t),
	\end{align*}
	as an integral sum (since $f$ is absolutely continuous and $f(0)=0$). 
\end{proof}

Similarly,
\begin{equation*}
	(Q\text{1\hspace{-0.55ex}I})_{n+2}= \frac1{\log 2} \sum_{k=2^n+1}^{2^{n+1}}f'(a_{n,k-(2^n+1)}) \cdot \frac1k\to \int_0^1 f'(s) \, ds = f(1)=1,
\end{equation*}
as an integral sum.

%Note that $Q_s$ is a positive operator, that it maps $c_0$ into itself and that $Q_s\emm=\emm$.

Consider operator $R$ on~$\ell_\infty$ given by
$$
(Rx)_1 = x_1, (Rx)_2 = x_2 \ \text{and} \ (Rx)_k=x_{n+2}, \ 2^n+1 \leqslant k \leqslant 2^{n+1}, \ n=1, 2, \dots.$$

\begin{prop}
	Let $f\in \Gamma$ be an absolutely continuous function.
	Let $A$ be a weak$^*$ compact convex subset of $EL$, such that operators $Q^*$ and $R^*$ map $A$ into itself. We have 
	$${\rm card}\! \ ({\rm ext}\! \ A(f) ) \geqslant {\rm card}\! \ ({\rm ext}\! \ A).$$
\end{prop}

\begin{proof}
	Since $A$ is weak$^*$ compact and convex, it is easy to check that $A(f)$ is weak$^*$ compact convex subset of $A$.
	
	We apply Lemma~\ref{lem1} for  $M=A(f)$, $D=A$, $f_1 = Q^*$ and $f_2=R^*$. The condition (i) of Lemma~\ref{lem1} is clearly satisfied.
	Further, we have
	\begin{align*}
		\left|(QRx)_{n+2}-x_{n+2}\right|&= \left| \left(\sum_{k=2^n+1}^{2^{n+1}}f'(a_{n,k-(2^n+1)}) \cdot \frac1{k\log 2}\right) x_{n+2}-x_{n+2}\right|\\
		&\le\left| \sum_{k=2^n+1}^{2^{n+1}}f'(a_{n,k-(2^n+1)}) \cdot \frac1{k\log 2}-1\right|\|x\|\\
		&\mathop{\rightarrow}\limits_{n\to\infty} \left|\int_0^1 f'(s) \, ds -1\right|\|x\|= \left|f(1) -f(0) -1\right|\|x\|=0,
	\end{align*}	
	since $f\in \Gamma$.
	Thus,
	$QR - I :\ell_\infty \to c_0$. 
	
	For every $\omega \in A$, we have $\omega\circ Q \in A$. Moreover, 
	$$\gamma(\omega\circ Q,t)=\gamma(Q y_t)=f(t),$$
	by Lemma \ref{justification}. 
	Hence, $\omega \circ Q \in A(f)$. Also, $\omega(QR - I)=0$ for every $\omega \in A$. Therefore, condition (ii) is satisfied.
	Condition (iii) is satisfied by assumption.
	The assertion follows from Lemma~\ref{lem1}. 
\end{proof}

\begin{cor}
	The cardinality of $EL(f)$ 
	equal to $2^\mathfrak{c}$ for every absolutely continuous $f\in \Gamma$.
\end{cor}

\begin{proof}
	It is easy to see that $Q^*$ and $R^*$ map $EL$ into itself. The result follows from the fact that ${\rm card}\! \ ({\rm ext}\! \ EL)=2^\mathfrak{c}.$ 
\end{proof}

\begin{cor}
	The cardinality of $\mathfrak B(f)$
	equal to $2^\mathfrak{c}$ for every absolutely continuous $f\in \Gamma$ with H\"older continuous derivative.
\end{cor}

\begin{proof}
	The result follows from Lemma~\ref{lem2} below and the fact that the cardinality of ${\rm ext}\! \ \mathfrak B$ is $2^\mathfrak{c}.$ 
\end{proof}

\begin{lem}\label{lem2} 
	Operator $R^*$ maps $\mathfrak B$ into itself. If $f\in \Gamma$ is absolutely continuous with H\"older continuous derivative, then the operator $Q^*$ maps $\mathfrak B$ into itself.
\end{lem}

\begin{proof}
	Let $B\in \mathfrak B$. It is clear that $B\circ Q$ and $B\circ R$ are positive and normalised.
	
	1. For every $n\in \mathbb N$ consider
	$$(Q(x-Tx))_{n+2} =\sum_{k=2^n+1}^{2^{n+1}} w_{n,k} (x_k-x_{k+1})$$
	$$=w_{n,2^n+1} x_{2^n+1}+\sum_{k=2^n+2}^{2^{n+1}-1} x_k (w_{n,k}-w_{n,k-1})- w_{n,2^{n+1}-1} x_{2^{n+1}}.$$
	Since $f'$ is H\"older continuous, then it is bounded on $[0,1]$.
	Hence, 
	$$w_{n,2^n+1}= \frac{f'(a_{n,0})}{2^n+1} \to 0\ \text{and} \ w_{n,2^{n+1}-1}= \frac{f'(a_{n,2^n-2})}{2^{n+1}-1} \to 0, \ n \to \infty.$$
	%		$$w_{n,2^n+1}= \frac{f'(a_{n,0})}{2^n+1}=\frac{f'(0)}{2^n+1} \to 0, \ n \to \infty,$$
	%		$$w_{n,2^{n+1}-1}= \frac{f'(a_{n,2^n-2})}{2^{n+1}-1} \to 0, \ n \to \infty,$$
	%		since 
	%		$$a_{n,2^n-2}= \frac1{\log 2} \sum_{i=2^n+1}^{2^{n_1}-2} \frac1i \to 1, \ n \to \infty.$$
	Further,
	
	\begin{align*}
		\sum_{k=2^n+2}^{2^{n+1}-1} &|w_{n,k}-w_{n,k-1}|=\sum_{k=2^n+2}^{2^{n+1}-1} \left|\frac{f'(a_{n,k-(2^n+1)})}{k}-\frac{f'(a_{n,k-1-(2^n+1)})}{k-1}\right|\\
		&\leqslant \sum_{k=2^n+2}^{2^{n+1}-1} \frac1{k-1}\left|f'(a_{n,k-(2^n+1)})-f'(a_{n,k-1-(2^n+1)})\right|\\
		&+\sum_{k=2^n+2}^{2^{n+1}-1} \frac{\left|f'(a_{n,k-(2^n+1)})\right|}{k(k-1)}
	\end{align*}

	Since $f'$ is H\"older continuous, then
	\begin{align*}
		\left|f'(a_{n,k-(2^n+1)})-f'(a_{n,k-1-(2^n+1)})\right|&\le C |a_{n,k-(2^n+1)}-a_{n,k-1-(2^n+1)}|^\alpha\\
		&= o(n^{-\alpha}), \ n \to \infty
	\end{align*}
	by \eqref{partitionmesh}. Thus,
	$$\sum_{k=2^n+2}^{2^{n+1}-1} |w_{n,k}-w_{n,k-1}|\leqslant o(n^{-\alpha}) \sum_{k=2^n+2}^{2^{n+1}-1} \frac1{k-1} +O\left(\sum_{k=2^n+2}^{2^{n+1}-1} \frac{1}{k(k-1)}\right),$$
	which tends to zero as  $n \to \infty$.
	Hence,
	$Q(x-Tx)\in c_0.$
	Thus,
	$(B\circ Q)(x-Tx)=0$ for every $B\in \mathfrak B$. Thus, $B\circ Q \in \mathfrak B$ for every $B\in \mathfrak B$.

	2. For every $x\in \ell_\infty$ we have
	\begin{align*}
		R(x-Tx)&=(x_1-x_2; 0, x_2-x_3; 0, 0, 0, x_3-x_4; \dots)\\
		&=\sum_{n=0}^\infty (x_{n+1}-x_{n+2})\chi_{\{2^{n+1}-1\}} \in ac_0.
	\end{align*}
Thus,
	$(B\circ R)(x-Tx)=0$ and $B\circ R \in \mathfrak B$ for every $B\in \mathfrak B$. 
\end{proof}

% ------------------------------------------------------------------------

\begin{thebibliography}{99} %{99}
	
	\bibitem{dodds1998symmetric}
	P.~G. Dodds, B.~de~Pagter, E.~M. Semenov, and F.~A. Sukochev, Symmetric functionals and singular traces, Positivity. 1998. Vol. 2. No. 1. P. 47--75.
	
	\bibitem{semenov2015banach}
	Semenov, E.~M., Sukochev, F.~A., Usachev, A.~S., Zanin, D.~V, Banach limits and traces on $\mathcal{L}_{1,\infty}$, Adv. Math. 2015. Vol. 285. P. 568--628. 
	
	\bibitem{mazur1929metodach}
	Mazur S. O metodach sumowalnosci // Ann. Soc. Polon.
	Math.(Suppl.). 1929. P. 102--107.
	
	\bibitem{banah2001theory}
{\ S. Banach},
\newblock {\em Th\'eorie des op\'erations lin\'eaires},
\newblock \'Editions Jacques Gabay, Sceaux, 1993.
\newblock Reprint of the 1932 original.
	
	\bibitem{lorentz1948contribution}
	Lorentz G.G. A contribution to the theory of divergent sequences //
	Acta mathematica. 1948. Vol. 80. No. 1. P. 167--190.
	
	\bibitem{agnew1938extensions}
	Agnew R.P., Morse A.P. Extensions of linear functionals, with applications to limits,
	integrals, measures, and densities // Annals of Mathematics. 1938. Vol. 39. No. 1. P. 20--30.
	
	\bibitem{eberlein1950banach}
	Eberlein W.F. Banach--Hausdorff limits // Proceedings of the
	American Mathematical Society. 1950. Vol. 1. No. 5. P. 662--665.
	
	\bibitem{semenov2010invariant}
	Semenov E.M., Sukochev F.A. Invariant Banach limits and applications //
	Journal of Functional Analysis. 2010. Vol. 259. No. 6. P. 1517--1541.
	
	\bibitem{aliprantis2006positive}
	Aliprantis C.D., Burkinshaw O. Positive operators // Academic Press. 1985. 376 p. 
	
	\bibitem{chou1969size}
	Chou C. On the size of the set of left invariant means on a semigroup //
	Proceedings of the American Mathematical Society. 1969. Vol. 23.
	No. 1. P. 199--205.
	
	\bibitem{semenov2023banach}
	Semenov E., Sukochev F., Usachev A. Banach limits and
	their applications //
	Notices Amer. Math. Soc. 2023. 70(6).  918--925. 
	
	\bibitem{sofi2021banach}
	Sofi M. A. Banach limits: Some new thoughts and perspective //
	The Journal of Analysis 2021. 29.  591--606. 
	
	\bibitem{semenov2020geom}
	{ Semenov, E.~M., Sukochev, F.~A., and Usachev, A.~S.}
\newblock Geometry of {B}anach limits and their applications.
\newblock { Russian Math. Surveys, 75}, 4 (2020), 153--194.
	
	\bibitem{das2021banach}
	Das G., Nanda S. Banach Limit and Applications // CRC Press. 2021. 216 p.
	
	\bibitem{semenov2023banach1}
	Avdeev, N., Semenov, E., Usachev, A. S., Zvolinskiy R. The set of Banach limits and its discreete and continuous subsets, Doklady Math., in press.
	
	\bibitem{semenov2019main}
	Semenov, E. M., Sukochev, F. A., Usachev, A. S. The main classes of invariant Banach limits, Izv. Math. 83:1 (2019), 124--150
	
	\bibitem{alekhno2015banach}
	Alekhno E. A. On Banach--Mazur limits // Indagationes Mathematicae,
	2015, Vol. 26, No. 4, P. 581--614.	
	
	\bibitem{semenov2014geom}
{ E.~M. Semenov, F.~A. Sukochev, A.~S. Usachev},
\textit{ Geometric properties of the set of {B}anach limits},
\newblock { Izv. Ross. Akad. Nauk Ser. Mat. } {78} (2014), 177--204.
	
	\bibitem{natanson1957}  
	Natanson I. P. The theory of functions of a real variable //
	Moscow, 1957. 
	
	\bibitem{alehno2016por}
{ E. Alekhno, E. Semenov, F. Sukochev, A. Usachev},
\newblock  \textit{Order and geometric properties of the set of Banach
	limits},
\newblock { St. Petersburg Math. J.}  (2016),  3--35.
	
	\bibitem{sucheston1967banach}
	Sucheston L. Banach limits // The American Mathematical
	Monthly. 1967. Vol. 74. No. 3. P. 308--311.
	
	\bibitem{semenov2013extreme}
	Semenov E.M., Sukochev F.A. Extreme points of the set of Banach limits //
	Positivity. 2013. Vol. 17. No. 1. P. 163--170.
	
	\bibitem{klee1960asymp}
	Klee V. Asymptotes and projections of convex sets // Math. Scand. 1960. Vol.8. 356--362.
	
	\bibitem{alehno2018invariant}
	Alekhno E., Semenov E., Sukochev F., Usachev A. Invariant Banach limits and their extreme points // Studia Math.  2018. Vol.242(1). 79--107.
	
\end{thebibliography}
\end{document}